\documentclass[hyperref, 12pt]{article}
\usepackage{graphicx} 
\usepackage{float} 
\usepackage{amssymb}
\usepackage{amsmath}
\usepackage{mathtools}
\usepackage[thmmarks,amsmath]{ntheorem} 
\usepackage{bm} 
\usepackage[colorlinks,linkcolor=red ,citecolor=cyan]{hyperref}
\usepackage{extarrows}
\usepackage[hang,flushmargin]{footmisc} 
\usepackage[square,comma,sort&compress,numbers]{natbib} 
\usepackage{mathrsfs} 
\usepackage[font=footnotesize,skip=0pt,textfont=rm,labelfont=rm]{caption,subcaption} 
\usepackage{booktabs} 
\usepackage{tocloft}

{
	\theoremstyle{nonumberplain}
	\theoremheaderfont{\bfseries}
	\theorembodyfont{\normalfont}
	\theoremsymbol{\mbox{$\Box$}}
	\newtheorem{proof}{Proof.}
}

\makeatletter
\@addtoreset{equation}{section}
\makeatother

\usepackage{theorem}
\newtheorem{theorem}{Theorem}[section]
\newtheorem{lemma}{Lemma}[section]
\newtheorem{definition}{Definition}[section]

\newtheorem{corollary}{Corollary}[section]
\newtheorem{proposition}[theorem]{Proposition}
{
	\theoremheaderfont{\bfseries}
	\theorembodyfont{\normalfont}
	
}
\graphicspath{{figure/}}                                 
\usepackage[a4paper]{geometry}                           
\providecommand{\keywords}[1]
{
	\textbf{\text{Keywords: }} #1
}
\begin{document}
	\title{\bf Asymptotic speeds of spreading for the Lotka-Volterra system with strong competition in $\mathbb{R}^N$} 
	\date{}
	\author{ Hui Bao$^{\dagger}$, Hongjun Guo$^{\dagger}$ \\
		\footnotesize{$^{\dagger}$ School of Mathematical Sciences, Key Laboratory of Intelligent Computing and Applications}\\
		\footnotesize{ (Ministry of Education), Institute for Advanced Study, Tongji University, Shanghai, China}\\}

\renewcommand{\thefootnote}{\fnsymbol{footnote}}
\maketitle

\begin{abstract}
This paper is concerned with the asymptotic 
	spreading behavior of solutions of the Lotka-Volterra 	system with strong competition in $\mathbb{R}^{N}$. Two types of initial conditions are proposed: (C1) two species initially occupy bounded domains; (C2) two species initially occupy the whole space separately. The spreading dynamics for (C1) (C2) is strongly depending on the speeds of traveling fronts of the scalar equations with no competition and the system.  We give the asymptotic speeds of spreading for both (C1) (C2).
\end{abstract}

\section{Introduction}
In this paper, we consider the Cauchy problem of the Lotka-Volterra competition-diffusion system in $\mathbb{R}^{N}$, that is, 
\begin{equation}\label{pro1}
	\begin{cases}
		u_{t}=d\Delta u+ru(1-u-av), &t>0,\, x\in \mathbb{R}^{N} ,\\
		v_{t}=\Delta v+v(1-v-bu) ,  &t>0,\, x\in \mathbb{R}^{N},\\
	\end{cases}
\end{equation}
with the initial value
\[u(0, x)=u_{0}(x), ~v(0,x )= v_{0}(x),\quad \text{for}~x\in \mathbb{R}^{N}.
\]
Precise conditions on $u_0(x)$, $v_0(x)$ will be given later.
In the system, $u(t, x)$, $v(t, x)$ represent the population 
densities of two competing species at the 
time $t$ and the 
space position $x$ respectively;
$d$ stands for the diffusion rate of $u$; $r$ represents
the intrinsic growth rate of $u$; $a$ and $b$ are the
competition coefficients
for two species respectively. All the 
parameters are assumed to be positive. Such a system is well-known in describing the dynamics of  competing species.

When there is no competition, i. e. if $u=0$ or $v=0$, 
the system reduces into the following Fisher-KPP equation (which is named by the pioneering works of Fisher \cite{FM} and Kolmogorov, Petrovsky and Piskunov \cite{KPP})
\begin{equation}\label{eq:FKPP}
	\begin{cases}
		w_{t}=D\Delta w+\rho w(1-w), &t>0,\, x\in \mathbb{R}^{N} ,\\
		w(0, x)=w_{0}(x),  &x\in \mathbb{R}^{N},\\
	\end{cases}
\end{equation}
where $D$, $\rho$ are two positive parameters. By recalling classical results of Aronson and Weinberger \cite{Aro},  it is known that there exists a minimal speed $c^{*}=2\sqrt{D\rho}$ such that the equation \eqref{eq:FKPP} exists a traveling front $\phi(x\cdot e-ct)$ for $t\in\mathbb{R}$, $x\in\mathbb{R}^N$ and every $e\in\mathbb{S}^{N-1}$ if and only if $c\ge c^*$. The profile $\phi(\xi)$ and the speed $c$ of the traveling front satisfy the following equation
\begin{equation}\label{into1}
	\begin{cases}
		D\phi''+c\phi'+\rho \phi(1-\phi)=0,  \, ~\text{in}~\xi \in \mathbb{R}
		\\
		\phi(-\infty)=1,\, \phi(+\infty)=0.
	\end{cases}
\end{equation}
By further referring to the results of \cite{Aro},  the solution $w(t,x)$ of \eqref{eq:FKPP} has an asymptotic speed of spreading which is equal to $c^*$ if $w_{0}(x)$ is compactly supported, that is, the following properties hold
\begin{equation*}
	\begin{aligned}
		&\lim_{t\to +\infty}\sup_{|x|\le ct}|1-w(t, x)|=0,
		~\text{for any}~ 0<c< c^{*},\\
		&\lim_{t\to +\infty}\sup_{|x|\ge ct}|w(t, x)|=0,
		~\text{for any}~ c>c^{*}.
	\end{aligned}
\end{equation*}
In the absence of the species $v$ (resp. $u$), let $c_{u}$ (resp. $c_{v}$) be the minimal speed 
associated with parameters $(D, \rho )=(d, r)$ 
(resp. $(D, \rho) = (1, 1)$) for system 
(\ref{into1}).  Then, $c_{u}=2\sqrt{dr}$ and $c_{v}=2$. For convenience, we denote the 
traveling front with the minimal speed of the Fisher-KPP equation
in absence of $v$ (resp. $u$) by $\phi(x\cdot e- c_{u}t)$
(resp. $\psi(x\cdot e- c_{v}t)$).

When the interaction of species (such as competition) is under consideration, the dynamics of solutions of the system become complicated. 
For instance, Iida, Lui and Ninomiya \cite{LN} considered stacked invasion waves in cooperative
systems of $N$-species with equal diffusion coefficients.  They found that species develop into
stacked fronts and spread at different speeds 
under certain conditions.
Besides, Girardin and Lam \cite{GL} considered
the  strong-weak competition case $(a<1<b)$ of (\ref{pro1}) with initial values being null or exponentially decaying in a right half-line. They obtained that 
if the weaker
competitor is also the faster one, then it is able to evade the stronger and slower competitor
by invading first into unoccupied territories. Another interesting discovery is that they found the acceleration phenomena during the period of invasion in some cases. 
Moreover,  Lin and Li \cite{LLi} had taken the weak competition $(a, b<1)$ of (\ref{pro1}) with compactly supported initial values into consideration and they got the spreading speed of the faster species and some estimates for the speed of the slower species.

For the strong competition (bistable) case $(a,\, b>1)$,  Carr\`{e}re \cite{Car} considered when the two strong 
competition species 
are initially absent from the
right half-line $x > 0$, and the slow one dominates
the fast one on $x < 0$. She found that the fast one will invade the
right space at its Fisher-KPP speed, and the slow one will be replaced by or will invade the fast one, depending on the parameters, at a slower speed.  Then, Peng, Wu and Zhou in \cite{PWZ} improved the results of  Carr\`{e}re \cite{Car} and even derived the sharp estimates of the 
spreading speeds based on the technology of \cite{Ham}.

We mention here that the results of Carr\`{e}re \cite{Car} and Peng, Wu and Zhou \cite{PWZ} are for the one-dimensional space. In this paper, we consider the high dimensional case ($N\ge 2$) and focus on the strong competition case as well, that is, we assume
\begin{itemize}
	\item[\textbf{(A1)}] ~(strong competition) the
	competition coefficients $a$, $b>1$.
\end{itemize}
Then, the system (\ref{pro1}) has
two stable steady equilibria $(0, 1)$ and $(1, 0)$.  The traveling front of the system \eqref{pro1} will play a crucial role in the analysis of the sequel. For the one-dimensional system $(N=1)$ of \eqref{pro1}, we refer to works of Gardner \cite{GE} and Kan-On \cite{Kan} for the existence of the traveling front connecting $(0, 1)$ and $(1, 0)$ which can be trivially extended to the high dimensional spaces. Precisely, for the high dimensional case $(N\ge 2)$, there is a unique (up to shifts) traveling front of (\ref{pro1})  with the form
$(u, v)(t,x)=(\Phi(x\cdot e-c_{uv}t),  \Psi(x\cdot e-c_{uv}t))
$ for every $e\in\mathbb{S}^{N-1}$.  The speed $c_{uv}$ satisfies
\[c_{uv}\in (-2, 2\sqrt{dr})
\]
and the profiles $\Phi(\xi)$, $\Psi(\xi)$ satisfy
\begin{equation}\label{pro2}
	\begin{cases}
		d\Phi''+c_{uv}\Phi'+
		r\Phi(1-\Phi-a\Psi)=0,~ &\xi \in \mathbb{R},\\
		\Psi''+c_{uv}\Psi'+\Psi(1-\Psi-b\Phi)=0,~
		&\xi \in \mathbb{R},\\
		\left( \Phi, ~\Psi\right)(-\infty)=(1, 0), ~ 
		\left( \Phi, ~\Psi\right)(+\infty)=(0, 1),\\
		\Phi'<0, ~\Psi'>0,~&\xi \in \mathbb{R}.
	\end{cases}
\end{equation}
One can easily notice that $c_{uv}<c_{u}$. More properties for the profiles $\Phi$, $\Psi$ can be referred to \cite{Kan,KF}.

We also  assume that 
\begin{itemize}
	\item[\textbf{(A2)}] ~the speed $c_{uv}>0$.
\end{itemize}
This assumption means that if we only look at the traveling front $(\Phi(x\cdot e-c_{uv}t),  \Psi(x\cdot e-c_{uv}t))$, the species $u$ is the winner in the competition since its territory, namely the territory of $(1,0)$, is expending at the speed $c_{uv}$ as $t\rightarrow +\infty$. Some sufficient conditions on parameters to ensure $c_{uv}>0$ can be referred to \cite{GLi,Kan} and we also refer to \cite{Gir,GN,MZO,RM} for some related discussions.

We set the initial values $u_0(x)$ and $v_0(x)$ satisfy 
\begin{eqnarray}\label{Initial}
	u_0(x)=\left\{\begin{array}{lll}
		1, &&\hbox{ in } U,\\
		0, &&\hbox{ in } \mathbb{R}^N\setminus U,
	\end{array}
	\right.
	\quad
	v_0(x)=\left\{\begin{array}{lll}
		1, &&\hbox{ in } V,\\
		0, &&\hbox{ in } \mathbb{R}^N\setminus V,
	\end{array}
	\right.
\end{eqnarray}
where $U$ and $V$ are measurable sets. 
For convenience, we assume $U\cap V=\emptyset$. Then, by the comparison principle, one has that $0\le u(t,x)\le 1$ and $0\le v(t,x)\le 1$ for $t\ge 0$ and $x\in\mathbb{R}^N$. 
In this paper, we consider two scenarios:
\begin{itemize}
	\item[\textbf{(C1)}] both $U$ and $V$ are compact sets.
	
	\item[\textbf{(C2)}] $U$ is a general measurable set and $V=\mathbb{R}^N\setminus U$;
\end{itemize}
The first scenario \textbf{(C1)} can be generalized to the case that both $u_0(x)\ge 0$ and $v_0(x)\ge 0$ are compactly supported, corresponding to the second scenarios of \cite{PWZ} for the one-dimensional case, in the cost of some further arguments as in \cite{PWZ}. The second scenario \textbf{(C2)} can also be generalized to more general initial conditions. But we prefer to keep these notations for simplicity of the presentation and coincidence with our forthcoming paper concerning the Fredlin-G\"{a}rtner formula of spreading speeds for general measurable sets $U$ and $V$.

As in \cite{PWZ}, we assume that the species $u$ always spreads sucessfully for scenario \textbf{(C1)} in the sense that
\begin{itemize}
	\item[\textbf{(A3)}] $\lim_{t\rightarrow +\infty} (u,v)(t,x)=(1,0)$ locally uniformly in $\mathbb{R}^N$.
\end{itemize}
This assumption may fail when the initial support $U$ of $u$ is small. But it holds for both \textbf{(C1)} and \textbf{(C2)} when $U$ is large, that is, the following proposition. 

\begin{proposition}\label{proposition1}
	There is $\rho>0$ such that the solution $(u,v)(t,x)$ of \eqref{pro1} associated with $U=B(0,\rho)$ and $V$ being any set with $U\cap V=\emptyset$ satisfies \textbf{(A3)}.
\end{proposition}

%
%
%

Then, we first consider the initial 
condition \textbf{(C1)} which can be interpreted as
both $u$ and $v$ are invasive
species that initially occupy only bounded domains. In this case, 
the speeds $c_{u}$ and $c_{v}$ are key roles to determine the spreading dynamics of solutions.

When $c_{u}>c_{v}$, the following result reveals that $u$ will always  be in front of $v$ and the 
competition winner assumption \textbf{(A2)}
derives that
$v$ will extinct in the whole space
$\mathbb{R}^{N}$ in the long-run.

\begin{theorem}\label{thm2}
	Assume that \textbf{(A1)}-\textbf{(A3)} and \textbf{(C1)}
	hold. If $c_{u}>c_{v}$,
	then there holds that
	\begin{equation*}
		\begin{aligned}
			&\lim_{t\to +\infty}\sup_{|x|\le ct}
			|u(t, x)-1| =0,\quad \text{for}~
			0<c<c_{u},\\
			&\lim_{t\to +\infty}\sup_{|x|\ge ct}
			|u(t, x)| =0,\quad \text{for}~
			c>c_{u},
		\end{aligned}
	\end{equation*}
	and 
	\begin{equation*}
		\lim_{t\to +\infty}\sup_{x\in\mathbb{R}^N} |v(t, x)|=0.
	\end{equation*}
\end{theorem}

We then turn to the other case $c_{v}>c_{u}$, namely $c_{v}>c_{u}>c_{uv}$. In this case, since
the speed of $v$ is faster than the speed of $u$,  $v$ will spread at the speed $c_{v}$ since there is no competition (or very low competition) at the spreading front of $v$ and
the competition arises far behind the spreading front of $v$.  Moreover, the competition will lead that $u$ spreads at the slower speed $c_{uv}$.

\begin{theorem}\label{thm3}
	Assume that \textbf{(A1)}-\textbf{(A3)} and \textbf{(C1)}
	hold. 
	If $c_{v}>c_{u}$, then there holds that
	\begin{equation*}
		\begin{aligned}
			&\lim_{t\to +\infty}\sup_{|x|\le ct}
			\left\lbrace |u(t, x)-1| +
			|v(t, x)|\right\rbrace =0,\quad \text{for}~
			0<c<c_{uv},\\
			&\lim_{t\to +\infty}\sup_{c_{1}t\le|x|\le c_{2}t}
			\left\lbrace |u(t, x)| +|v(t, x)-1|\right\rbrace 
			=0,\quad\text{for}~
			c_{uv}<c_{1}\le c_{2}<c_{v},\\
			&\lim_{t\to +\infty}\sup_{|x|\ge ct}
			\left\lbrace |u(t, x)| +
			|v(t, x)|\right\rbrace =0,\quad\text{for}~
			c>c_{v}.
		\end{aligned}
	\end{equation*}
\end{theorem}



Then, we consider the scenario \textbf{(C2)}. The biological interpretation of this condition is that $u$ and $v$ initially occupy the whole space $\mathbb{R}^N$ separately. In this case, the set $U$ is assumed to be a general measurable set which could be bounded or unbounded. When $U$ is bounded, one can expect from results of the one-dimensional case and the scalar equation that the solution has an asymptotic speed independent of $U$. However, when $U$ is unbounded, the spreading of the solution keeps a memory of the set $U$ in some sense, see the work of Hamel and Rossi \cite{HR} for the scalar equation. We borrow some notions from \cite{HR}. 

\begin{definition}\cite{HR}\label{deuv}
The set of bounded directions of  $U$ and the set of unbounded directions of $U$ are given by	
	\begin{equation*}
		\mathcal{B}(U):=\left\lbrace 
		\xi\in \mathbb{S}^{N-1}:\liminf_{\tau\to
			+\infty}\frac{\text{dist}
			(\tau \xi , U)}{\tau}>0\right\rbrace 
	\end{equation*}
	and 
	\begin{equation*}
		\mathcal{U}(U):=\left\lbrace 
		\xi\in \mathbb{S}^{N-1}:\lim_{\tau\to
			+\infty}\frac{\text{dist}
			(\tau \xi , U)}{\tau}=0 \right\rbrace,
	\end{equation*}
	respectively.
\end{definition}

Roughly speaking, the set $U$
	is bounded around bounded directions and  unbounded around unbounded
		directions. The sets $\mathcal{B}(U)$ and $\mathcal{U}(U)$ are 
respectively open and closed relatively to $\mathbb{S}^{N-1}$.
The condition
$\xi \in \mathcal{B}(U)$ is equivalent to 
the existence of a open cone   $\mathcal{C}$ 
containing the ray 
$\mathbb{R}^{+}{\xi}=\left\lbrace \tau \xi:
\tau>0\right\rbrace $
such that $U\cap\mathcal{C}$ is bounded, namely, 
$\mathbb{R}^{+}\xi \subset \mathcal{C}\subset
(\mathbb{R}^{N}\setminus U)\cup B_{R}$ for some $R>0$.
Conversely, for any $\xi \in \mathcal{U}(U)$ indicates that 
for any $\mathcal{C}$ containing the ray 
$\mathbb{R}^{+}{\xi}=\left\lbrace \tau \xi:
\tau>0\right\rbrace $, there holds $U\cap\mathcal{C}$ 
is unbounded.
We also define the notion of
positive distance-interior
$U_{\rho}$ (with $\rho>0$) 
of the set $U$ as
\[
U_{\rho}:=\left\lbrace x\in U: 
\text{dist}(x, \partial U)\ge\rho \right\rbrace. 
\]

Now we present spreading results for the initial condition
\textbf{(C2)}. 

\begin{theorem}\label{thm2.4}
	Assume  \textbf{(A1)}-\textbf{(A2)} and \textbf{(C2)} hold.
	Assume that the initial support $U$
	satisfies $U_{\rho}\neq \emptyset$ and 
	\begin{equation}\label{ABcon}
		\mathcal{B}(U)\cup
		\mathcal{U}(U_{\rho})=
		\mathbb{S}^{N-1},
	\end{equation}
	where $\rho$ is given by Proposition~\ref{proposition1}.
	Then for every $e\in \mathbb{S}^{N-1}$, there holds the following results:
	\begin{equation}\label{26}
		\begin{aligned}
			&\lim_{t\to +\infty}\sup_{0\le s\le c}\left| u(t, ste)-1
			\right| +\left| v(t, ste)\right| =0,\quad \text{for}~~0\le c<w(e),\\
			&\lim_{t\to +\infty}\sup_{ s\ge c}\left| u(t, ste)
			\right| +\left| v(t, ste)-1\right| =0,\quad \text{for}~~ c>w(e),
		\end{aligned}
	\end{equation}
	where $w(e)$ is given explicitly by the variational
	formula
	\begin{equation}\label{unuve}
		w(e)=\sup_{\xi\in \mathcal{U}
			(U), ~\xi\cdot e\ge 0}
		\frac{c_{uv}}{\sqrt{1-(\xi\cdot e)^{2}}}
	\end{equation}
	with the conventions
	\begin{equation}\label{unge}
		\begin{cases}
			w(e)=c_{uv}&\text{if there is no}~\xi \in 
			\mathcal{U}(U)~\text{such that}~\xi\cdot e\ge 0,\\
			w(e)=+\infty,&\text{if}~e\in \mathcal{U}(U).
		\end{cases}
	\end{equation}
\end{theorem}

As far as the uniformity of the convergence with respect to $e\in\mathbb{S}^{N-1}$ in Theorem~\ref{thm2.4}  is concerned, we have the following theorem.

\begin{theorem}\label{ASSset}
	Under the same assumptions in Theorem~\ref{thm2.4},  it holds that for any 
	compact sets $\mathcal{C}\subset \mathbb{R}^{N}$,
	\begin{equation}\label{28}
		\begin{aligned}
			&\lim_{t\to +\infty}\sup_{x\in \mathcal{C}
			}\left| u(t, tx)-1
			\right| +\left| v(t, tx)\right| =0~~\text{for}~~\mathcal{C}\subset
			\mathcal{W},\\
			&\lim_{t\to +\infty}\sup_{ x\in
				\mathcal{C}}\left| u(t, tx)
			\right| +\left| v(t, tx)-1\right| =0~~\text{for}~~ \mathcal{C}\subset
			\mathbb{R}^{N}\setminus \overline{\mathcal{W}},
		\end{aligned}
	\end{equation}
	where $\mathcal{W}$ is the envelop set
	defined as
	\begin{equation*}
		\mathcal{W}:=\left\lbrace re:e\in 
		\mathbb{S}^{N-1}, 0\le r<w(e)\right\rbrace,
	\end{equation*}
	which has the following expression
	\begin{equation}\label{expressionW}
		\mathcal{W}=\mathbb{R}^{+}\mathcal{U}
		(U)+B(0,c_{uv})
	\end{equation}
	(with the convention that
	$\mathbb{R}^{+}\emptyset
	+B(0,c_{uv})=B(0,c_{uv})$).
\end{theorem}

We refer to \cite{HR} for some sufficient conditions and examples on \eqref{ABcon} and  further comments and counter-examples to Theorems~\ref{thm2.4}, \ref{ASSset} when \eqref{ABcon} is not satisfied.

We organize this paper as follows. In Section 2, we  consider two types of initial value problems and deduce two key lemmas. 
In Section \ref{sec5}, we consider scenario \textbf{(C1)}, that is, we prove Theorems  \ref{thm2} and \ref{thm3}. Section \ref{sec4} is 
 devoted to prove Theorem \ref{thm2.4}
 and Theorem~\ref{ASSset} for scenario \textbf{(C2)}.

\section{Two key lemmas}
In this section, we consider two types of initial value problems whose related solutions  are sub and supersolutions for scenarios \textbf{(C1)} \textbf{(C2)}. Define
$$\delta_0:= \min\left\{\frac{1}{2(3+4\max\{a, b\})},\frac{a-1}{2(1+2b)a}, 
\frac{b-1}{2(4a+1)b},\frac{1}{6},\frac{a-1}{2(4b+1)a}, 
 \frac{b-1}{2(1+2a)b}\right\}.$$
Fix any $0<\delta\le \delta_0$.

We first consider  the following initial condition
\begin{eqnarray*}
	\left\{\begin{array}{lll}
		\left( u_{0}(x), v_{0}(x)\right)
		=(1-\delta, \delta), && \hbox{for } x\in B(0, \rho),\\
		 \left( u_{0}(x), v_{0}(x)\right)
		=(0, 1), &&  \hbox{for } x\in \mathbb{R}^{N}\setminus B(0, \rho),
	\end{array}
	\right.
\end{eqnarray*}
with $\rho>0$. Denote the related solution by $(\underline{u}_{\rho}(t,x),\overline{v}_{\rho}(t,x))$. Then, the following property holds for $(\underline{u}_{\rho},\overline{v}_{\rho})$.

\begin{lemma}\label{lemma:sub}
There is $\rho>0$ such that for any $\varepsilon\in (0,c_{uv})$, there holds that
\begin{equation}\label{extending}
	(\underline{u}_{\rho}(t,x),\overline{v}_{\rho}(t,x))\rightarrow (1,0)~\text{
		uniformly in}~\left\lbrace x\in \mathbb{R}^{N}; |x|\le 
	(c_{uv}-\varepsilon)t\right\rbrace ~\text{as}~t\to +\infty.
\end{equation}
\end{lemma}

\begin{proof}
We prove this lemma by constructing a subsolution.

{\it \textbf{Step 1}: choice of some parameters.} For $\varepsilon>0$, we  introduce a $C^{2}$ function $h_{\varepsilon}(r): [0,+\infty]\rightarrow \mathbb{R}$ as  in
\cite{HFB} satisfying the following properties:
\begin{equation}\label{hcon}
	\begin{cases}
		0\le h_{\varepsilon}'\le 1 ~\text{on}~[0, +\infty),\\
		h_{\varepsilon}'=0~\text{on a neighborhood of 0},\\
		h_{\varepsilon}(r)=r~\text{on}~[H_{\varepsilon}, +\infty)~\text{
			for some}~H_{\varepsilon}> 0,\\
		\max\{d,1\}\left(\frac{N-1}{r}h_{\varepsilon}'(r)+	h_{\varepsilon}''(r)\right)\le \frac{\varepsilon}{2}~~\text{on}~[0, +\infty).
	\end{cases}
\end{equation}
In particular,  it necessarily has 
\begin{equation}\label{hcon2}
	r\le h_{\varepsilon}(r)\le r+h_{\varepsilon}(0)~\text{for all }~r\ge 0.
\end{equation}
Since $\left( \Phi, ~\Psi\right)(-\infty)=(1, 0), ~ 
		\left( \Phi, ~\Psi\right)(+\infty)=(0, 1)$ and by \cite{Kan,KF} (see also \cite{MT}), there is $M>0$ such that
\begin{equation}\nonumber
		\begin{aligned}
			&1-\delta\le \Phi<1,\, 0<\Psi\le \delta~\text{for}~\xi \le -M, \quad 1-\delta\le \Psi<1,\, 0<\Phi\le \delta ~\text{for}~\xi \ge M.
		\end{aligned}
	\end{equation}
	and
	\[\Phi''<0, \, \Psi''>0 ~\text{for}~\xi \le -M,\quad \Phi''>0, \, \Psi''<0 ~\text{for}~\xi \ge M.
	\] 
	 There are positive
	constants $k_{1}$ and $k_{2}$ such that $-\Phi'\ge k_1,~\Psi'\ge k_2~\text{for}~|\xi|\le M.$
	Let $p_1(\xi)$ and $p_2(\xi)$ be decreasing and increasing $C^2$ functions respectively such that
	$$p_1(\xi)=2a \hbox{ for } \xi\le -M,\quad p_1(\xi)=1 \hbox{ for } \xi\ge M,$$ 
	and 
	$$p_2(\xi)=1 \hbox{ for } \xi\le -M,\quad p_1(\xi)=2b \hbox{ for } \xi\ge M.$$
	Even if it means increasing $M$, one can assume that 
	\begin{equation}\label{p}
	\max\{\|p_1'\|_{L^{\infty}},\|p_1''\|_{L^{\infty}},\|p_2'\|_{L^{\infty}},\|p_2''\|_{L^{\infty}}\}\le \min\left\{1,\frac{1}{(1+d)c_{uv}+d}, 
	\frac{1}{c_{uv}+2}\right\}.
	\end{equation}
	Take 
	\begin{equation}\label{mu}
	0<\delta\le \min\left\{\delta_0,\frac{k_1}{2},\frac{k_2}{2}\right\},\quad \mu=\min\left\{\frac{r}{4},\frac{1}{4},
	 \frac{1}{2}r(a-1),\frac{b}{2}\right\},
	\end{equation}
and $\omega>0$ large such that
	\begin{equation}\label{omega}
	\frac{1}{2}\omega \min\{k_1,k_2\} \ge \max\left\{
	 2 a +\frac{2a(b+1)r}{\mu}+1, 2b+\frac{2b(a+1)}{\mu}+1\right\}.
	\end{equation}
	Take $R\ge M+H_{\varepsilon} +\omega$. 
	Let $\rho\ge R+M$. 

{\it \textbf{Step 2}: construction of a subsolution.} For $t\ge 0$ and $x\in\mathbb{R}^N$, define 
\begin{equation}\label{eqsub}
	\begin{cases}
		\underline{u}(t, x)=\max\left\lbrace \Phi(\xi(t,x))-p_1(\xi(t,x))\delta e^{-\mu t},  0\right\rbrace,\\
		\overline {v}(t, x)=\min\left\lbrace\Psi(\xi(t,x))+p_2(\xi(t,x))\delta e^{-\mu t}, 1\right\rbrace,\\
	\end{cases}
\end{equation}
where 
\[\xi(t,x)=h_{\varepsilon}(|x|) -(c_{uv}-\frac{\varepsilon}{2})t -\omega \delta e^{-\mu t} +\omega \delta - R.
\]
We are going to show that $(\underline{u},\overline {v})$ is a subsolution.

We first check the initial values. For $t=0$ and $|x|< \rho$, one has that
$$\underline{u}(0,x)\le 1- \delta\le \underline{u}_{\rho}(0,x),\quad \overline{v}(0,x)\ge \delta\ge \overline{v}_{\rho}(0,x).$$
For $t=0$ and $|x|\ge \rho$, one has that $\xi(0,x)=h_{\varepsilon}(|x|)-R\ge \rho -R\ge M$ and
$$\underline{u}(0,x)\le \max\{\delta-  \delta,0\}=0\le \underline{u}_{\rho}(0,x),\quad \overline{v}(0,x)\ge 1-\delta+\delta \ge \overline{v}_{\rho}(0,x).$$

Then, we check that
$$N_1[\underline{u},\overline{v}]:=\underline{u}_{t}-d\Delta \underline{u}-r\underline{u}(1-\underline{u}-a\overline{v})\le 0,$$
for $t\ge 0$ and $x\in\mathbb{R}^N$ such that $\underline{u}(t,x)>0$ and
$$N_2[\underline{u},\overline{v}]:=\overline{v}_{t}-\Delta \overline{v} -\overline{v}(1-\overline{v}-b\underline{u})\ge 0,$$
for $t\ge 0$ and $x\in\mathbb{R}^N$ such that $\overline{v}(t,x)<1$.

After some computation and by \eqref{pro2},	 we obtain
	\begin{equation}\label{n1}
		\begin{aligned}
			N_{1}[\underline u, \overline v]
			=&(\frac{\varepsilon}{2} +\omega\delta\mu e^{-\mu t}) \Phi' + d(1-|h'_{\varepsilon}|^2) \Phi''
			-d\left(h^{''}_{\varepsilon}+\frac{N-1}{|x|}
		h^{'}_{\varepsilon} \right)\Phi' + p_1 \delta \mu e^{-\mu t}\\
		&+(c_{uv}-\frac{\varepsilon}{2} -\omega\delta\mu e^{-\mu t}) p_1'\delta e^{-\mu t} + d|h'_{\varepsilon}|^2 p_1''\delta e^{-\mu t}
			+d\left(h^{''}_{\varepsilon}+\frac{N-1}{|x|}
		h^{'}_{\varepsilon} \right)p_1'\delta e^{-\mu t} \\
		&+r\Phi(1-\Phi-a\Psi)-r\underline{u}(1-\underline{u}-a\overline{v}).
		\end{aligned}
	\end{equation}
For $\xi(t,x)\le -M$, one has that $p_1\equiv 2a$, $p_2\equiv 1$, $\Phi(\xi(t,x))\ge 1-\delta$, $\underline{u}(t,x)=\Phi(\xi(t,x))-2a\delta e^{-\mu t}\ge 1-(1+2a)\delta$, $\Psi(\xi(t,x))\le \delta$ and $\overline{v}(t,x)-\Psi(\xi(t,x))\le \delta e^{-\mu t}$. Then,
\begin{align*}
r\Phi(1-\Phi-a\Psi)-r\underline{u}(1-\underline{u}-a\overline{v})=& r\Phi(\underline{u}-\Phi+a\overline{v}-a\Psi) +r 2a \delta e^{-\mu t}(1-\underline{u}-a\overline{v})\\
\le& -r(1-\delta) a \delta e^{-\mu t} +r 2a \delta e^{-\mu t} (1+2a) \delta\\
\le& -\frac{1}{2}r a \delta e^{-\mu t},
\end{align*}
by $\delta\le \delta_0$. Combined with \eqref{hcon}, \eqref{mu}, $\Phi'<0$ and $\Phi''<0$ for $\xi\le -M$, it follows from \eqref{n1} that $N_{1}[\underline u, \overline v]\le 0$ for $\xi(t,x)\le -M$. 
	
For $\xi(t,x)\ge M$, one has that $p_1\equiv 1$, $p_2\equiv 2b$, $\Phi(\xi(t,x))\le \delta$, $\underline{u}(t,x)=\Phi(\xi(t,x))-\delta e^{-\mu t}\le \delta$, $\overline{v}(t,x)\ge \Psi(\xi(t,x))\ge 1-\delta$ and $\overline{v}(t,x)-\Psi(\xi(t,x))\le 2b\delta e^{-\mu t}$. Then,
\begin{align*}
r\Phi(1-\Phi-a\Psi)-r\underline{u}(1-\underline{u}-a\overline{v})=& r\Phi(\underline{u}-\Phi+a\overline{v}-a\Psi) +r   \delta e^{-\mu t}(1-\underline{u}-a\overline{v})\\
\le& 2ab r \Phi \delta e^{-\mu t} +r  \delta e^{-\mu t} (1-a+a\delta)\\
\le& -\frac{a-1}{2}r   \delta e^{-\mu t},
\end{align*}
by $\delta\le \delta_0$. Notice that $\xi(t,x)\ge M$ implies that $h_{\varepsilon}(|x|)\ge (c_{uv}-\frac{\varepsilon}{2})t   +  R+M-\omega \delta\ge H_{\varepsilon}$. Therefore, $h'_{\varepsilon}(|x|)=1$. Combined with \eqref{hcon}, \eqref{mu} and $\Phi'<0$, it follows from \eqref{n1} that $N_{1}[\underline u, \overline v]\le 0$ for $\xi(t,x)\ge M$. 

For $-M\le \xi(t,x)\le M$, one has that $-\Phi'(\xi(t,x))\ge k_1$. Notice that $\overline{v}(t,x)-\Psi(\xi(t,x))\le 2b\delta e^{-\mu t}$. Moreover,
\begin{align*}
r\Phi(1-\Phi-a\Psi)-r\underline{u}(1-\underline{u}-a\overline{v})=& r\Phi(\underline{u}-\Phi+a\overline{v}-a\Psi) +r p_1 \delta e^{-\mu t}(1-\underline{u}-a\overline{v})\\
\le& 2ab r \delta e^{-\mu t} +r 2a \delta e^{-\mu t}\\
\le&  2a(b+1) r\delta e^{-\mu t}.
\end{align*}
Notice that $\xi(t,x)\ge -M$ implies that $h_{\varepsilon}(|x|)\ge (c_{uv}-\frac{\varepsilon}{2})t   +  R-M-\omega \delta\ge H_{\varepsilon}$. Therefore, $h'_{\varepsilon}(|x|)=1$ and $h''_{\varepsilon}(|x|)=0$. Combined with \eqref{hcon}, \eqref{p}, \eqref{mu} and \eqref{omega}, it follows from \eqref{n1} that 
\begin{equation*}
		\begin{aligned}
			N_{1}[\underline u, \overline v]
			\le & -\omega k_1\delta\mu e^{-\mu t}  +2a \delta \mu e^{-\mu t} +2 a (b+1) r\delta e^{-\mu t}\\
			&+\|p'_1\|_{L^{\infty}} \delta \omega \delta\mu e^{-\mu t}+(c_{uv}\|p'_1\|_{L^{\infty}}+d\|p''_1\|_{L^{\infty}}+d \frac{\varepsilon}{2}\|p'_1\|_{L^{\infty}})\delta\mu e^{-\mu t}\\
			\le&  -\frac{1}{2}\omega k_1\delta\mu e^{-\mu t}+2a \delta \mu e^{-\mu t} +2 a (b+1) r\delta e^{-\mu t} +  \delta \mu e^{-\mu t}\le 0.
		\end{aligned}
	\end{equation*}
	
	In conclusion, $N_{1}[\underline u, \overline v]\le 0$ for $t\ge 0$ and $x\in\mathbb{R}^N$.
	
Similar computation can be applied to $N_2[\underline{u},\overline{v}]$. By \eqref{pro2}, one can compute that
	\begin{equation}\label{n2}
		\begin{aligned}
			N_{2}[\underline u, \overline v]
			=&(\frac{\varepsilon}{2}+\omega\delta \mu e^{-\mu t})\Psi ^{'}-
			\left( h^{''}_{\varepsilon}+\frac{N-1}{|x|}h_{\varepsilon}'\right)\Psi^{'}
			+\left( 1-|h_{\varepsilon}^{'}|^{2}\right)\Psi^{''}-p_{2}\delta \mu e^{-\mu t}\\
			&-\left( c_{uv}-\frac{\varepsilon}{2}-\omega\delta\mu e^{-\mu t}\right)
			p_{2}^{'}\delta e^{-\mu t}-|h_{\varepsilon}^{'}|^{2}p_{2}^{''}
			\delta e^{-\mu t}- \left( h_{\varepsilon}^{''}+
			\frac{N-1}{|x|}h^{'}_{\varepsilon}\right)p_{1}^{'}\delta e^{-\mu t}\\
			&+\Psi (1-\Psi -b\Phi)-\overline v(1-\overline v-b\underline u).  
		\end{aligned}
	\end{equation}
For $\xi(t, x)\le -M$, one has that $p_{1}\equiv2a,
p_{2}\equiv1$, $\Psi(\xi(t, x))\le \delta$, $\overline v(t, x)-\Psi(\xi(t, x)) =
-\delta e^{-\mu t}$, $\Phi(\xi(t, x))\ge 1-\delta $, $\underline{u}(t, x)=
\Phi(\xi(t, x))-p_{1}\delta e^{-\mu t}\ge1-\delta - 2a\delta$ and $\underline{u}(t, x)-\Phi(\xi(t, x))= -2a\delta^{-\mu t}$. 
Then,  
\begin{align*}
	\Psi(1-\Psi-b\Phi)-\overline{v}(1-\overline{v}-b\underline{u})
	=& \Psi(\overline{v}-\Psi+b\underline{u}-b\Phi) -\delta e^{-\mu t}(1-\overline{v}-b\underline{u})\\
	\ge& -\delta 2ab\delta e^{-\mu t}-\delta e^{-\mu t}(1-b(1-\delta
	 -2a\delta))\\
	\ge& \frac{b-1}{2}\delta e^{-\mu t},
\end{align*}
by $\delta \le \delta_{0}$. Combined with (\ref{hcon}), 
(\ref{mu}), $\Psi'>0$ and $\Psi''>0$ for
$\xi\le -M$, it follows from (\ref{n2}) that $N_{2}[\underline u, \overline v]\ge0$ for $\xi(t, x)\le -M.$

For $\xi(t, x)\ge M$, one has that $ p_{1}\equiv1$, $p_{2}\equiv2b,
\Psi(\xi(t, x))\ge 1-\delta$, $\underline{u}(t, x)-\Phi(\xi(t, x))=-\delta e^{-\mu t}$, 
$\overline v(t, x)=\Psi(\xi(t, x)) +p_{2}\delta e^{-\mu t}\ge 1-\delta $.Then,
\begin{align*}
	\Psi(1-\Psi-b\Phi)-\overline{v}(1-\overline{v}-b\underline{u})
	=& \Psi(\overline{v}-\Psi+b\underline{u}-b\Phi) -2b\delta e^{-\mu t}(1-\overline{v}-b\underline{u})\\
	\ge& (1-\delta)b\delta e^{-\mu t}-\delta 2b \delta e^{-\mu t}\\
	\ge& \frac{1}{2}b\delta e^{-\mu t},
\end{align*}
by $\delta \le \delta_{0}$, Combined with (\ref{hcon}), 
(\ref{mu}), $\Psi'>0$ and $\Psi''>0$ for
$\xi\le -M$, it follows from (\ref{n2}) that $N_{2}[\underline u, \overline v]\ge0$ for $\xi(t, x)\le -M.$

For $-M\le \xi(t,x)\le M$, one has that $\Psi'(\xi(t,x))\ge k_{2}$. It holds 
that $\underline u(t, x)-\Phi(\xi(t, x))\ge 2a\delta e^{-\mu t}$. Moreover, 
\begin{align*}
	\Psi(1-\Psi-b\Phi)-\overline{v}(1-\overline{v}-b\underline{u})
	=& \Psi(\overline{v}-\Psi+b\underline{u}-b\Phi) -p_{2}\delta e^{-\mu t}(1-\overline{v}-b\underline{u})\\
	\ge& -2ab\delta e^{-\mu t}-2b\delta e^{-\mu t}\\
	\ge& -2b(a+1)\delta e^{-\mu t},
\end{align*}
The inequality $\xi(t, x)\ge -M$ implies that $h_{\varepsilon}(|x|)
\ge H_{\varepsilon}$ and  therefore $h_{\varepsilon}'(|x|)=1, h_{\varepsilon}''
(|x|)=0$. Combined with \eqref{hcon}, \eqref{p}, \eqref{mu} and \eqref{omega}, it follows from \eqref{n1} that 
\begin{equation*}
	\begin{aligned}
		N_{2}[\underline u, \overline v]
		\ge & \omega k_2\delta\mu e^{-\mu t}  -2a \delta \mu e^{-\mu t} -2b
		 (a+1) \delta e^{-\mu t}\\
		&-\|p'_2\|_{L^{\infty}} \delta \omega \delta\mu e^{-\mu t}-(c_{uv}\|p'_2\|_{L^{\infty}}+\|p''_2\|_{L^{\infty}}+ \frac{\varepsilon}{2}\|p'_2\|_{L^{\infty}})\delta\mu e^{-\mu t}\\
		\ge&  \frac{1}{2}\omega k_2\delta\mu e^{-\mu t}-2b \delta \mu e^{-\mu t} -2 b (a+1) r\delta e^{-\mu t} -  \delta \mu e^{-\mu t}\ge 0.
	\end{aligned}
\end{equation*}

In conclusion, $N_{2}[\underline u, \overline v]\ge 0$ for $t\ge 0$ and $x\in\mathbb{R}^N$.

{\it \textbf{Step 3}: proof of \eqref{extending}.} By the comparison principle, one has that
$$\underline{u}_{\rho}(t,x)\ge \underline{u}(t,x) \hbox{ and } \overline{v}_{\rho}(t,x)\le \overline{v}(t,x) \hbox{ for $t\ge 0$ and $x\in\mathbb{R}^N$}.$$
Then, since $\xi(t,x)\le h_{\varepsilon}(0)-\frac{\varepsilon}{2} t -R\rightarrow -\infty$ as $t\rightarrow +\infty$ for $|x|\le (c_{uv}-\varepsilon) t$, \eqref{extending} immediately follows from \eqref{eqsub}.
\end{proof}

Proposition~\ref{proposition1} is a immediate consequence of Lemma~\ref{lemma:sub} and the comparison principle.

We then consider the following initial condition
\begin{equation*}
	\begin{cases}
		\left( u_{0}(x), v_{0}(x)\right)
		=(1, 0), &  \hbox{for } x\in \mathbb{R}^{N}\setminus B(0, R),\\
		\left( u_{0}(x), v_{0}(x)\right)
		=(\delta, 1-\delta), & \hbox{for } x\in B(0, R), 
	\end{cases}
\end{equation*}
with $R>0$. Denote the related solution by  $(\overline{u}_{R}(t,x),\underline{v}_{R}(t,x))$. Then, the following property holds for $(\overline{u}_{R},\underline{v}_{R})$.

\begin{lemma}\label{lem1}
	For any $\varepsilon>0$, there exists $R_{\varepsilon}>0$
	such that for any $R\ge R_{\varepsilon}$, 
	there holds 
	\begin{equation}\label{delta1}
		\begin{aligned}
			\overline{u}_{R}(t,x)\le 2\delta,~\underline{v}_{R}(t,x)(t, x)\ge 1-2\delta, ~&\text{for} ~0\le t\le 
			\frac{R-R_{\varepsilon}}{c_{uv}+\varepsilon}\\
			&\text{and} ~|x|\le R-R_{\varepsilon}-(c_{uv}+\varepsilon)t.
		\end{aligned}
	\end{equation}
	Moreover, there exists $T_{\varepsilon}>0$ such that 
	\begin{equation}\label{delta2}
		\begin{aligned}
			\overline{u}_{R}(t,x)\le \delta,~\underline{v}_{R}(t, x)\ge 1-\delta, ~&\text{for} ~T_{\varepsilon}\le t\le 
			\frac{R-R_{\varepsilon}}{c_{uv}+\varepsilon}\\
			&\text{and} ~|x|\le R-R_{\varepsilon}-(c_{uv}+\varepsilon)t.
		\end{aligned}
	\end{equation}
\end{lemma}
\begin{proof}
	The proof is based on the construction of a
	suitable supersolution.
	
	{\it \textbf{Step 1}: choice of some parameters.} As in the proofs of Lemma
	 \ref{lemma:sub}, one can also choose $M>0$ such that
	  \begin{equation}\nonumber
	  	\begin{aligned}
	  		&1-\delta\le \Phi<1,\, 0<\Psi\le \delta~\text{for}~\xi \le -M, \quad 1-\delta\le \Psi<1,\, 0<\Phi\le \delta ~\text{for}~\xi \ge M.
	  	\end{aligned}
	  \end{equation}
	  and
	  \[\Phi''<0, \, \Psi''>0 ~\text{for}~\xi \le -M,\quad \Phi''>0, \, \Psi''<0 ~\text{for}~\xi \ge M.
	  \] 
	  Furthermore, there are positive
	  constants $k_{1}$ and $k_{2}$ such that $-\Phi'\ge k_1,~\Psi'\ge k_2~\text{for}~|\xi|\le M.$ Moreover, one can take $M_\varepsilon \ge  M$ such that
	    \begin{equation*}
	  	\begin{aligned}
	  		&1-\frac{\delta}{2}\le \Phi<1,\, 0<\Psi\le \frac{\delta}{2}~\text{for}~\xi \le -M_\varepsilon, \quad 1-\frac{\delta}{2}\le \Psi<1,\, 0<\Phi\le \frac{\delta}{2} ~\text{for}~\xi \ge M_\varepsilon.
	  	\end{aligned}
	  \end{equation*}
	  
	  Consider the  associated functions $p_{i}(\xi), i=1,~2,~ \xi\in \mathbb{R}^{N}$, $h_{\varepsilon}(r): [0,+\infty]\rightarrow \mathbb{R}$ provided by Step 1 of Lemma \ref{lemma:sub}.
	  Moreover take 
	\begin{equation}\label{mu1}
	  	0<\delta\le \min\left\{\delta_0,\frac{k_1}{2},\frac{k_2}{2}\right\},\quad \mu=\min\left\{\frac{r}{4},\frac{1}{4},
	  	\frac{1}{2}r(a-1), \frac{b-1}{2}\right\},
	  \end{equation}
  $\omega>0$ large and
	$R\ge R_\varepsilon$ with $R_{\varepsilon}=\max\left\lbrace H_{\varepsilon}, 
	 h_{\varepsilon}(0)+\omega\delta +M+M_\varepsilon\right\rbrace $ such that \eqref{omega} holds and 
	 \begin{equation}\label{Rcon}
	 	\frac{\varepsilon (R+R_{\varepsilon})}{2(c_{uv}+\varepsilon)}>\omega\delta +M+M_\varepsilon+H_{\varepsilon}.
	 \end{equation}
	
{\it \textbf{Step 2}: construction of a supersolution.}	For $0\le t\le  \frac{R-R_{\varepsilon}}{c_{uv}+\varepsilon}$ and $x\in\mathbb{R}^N$, define 
\begin{equation}\label{eqsup}
	\begin{cases}
		\overline{u}(t, x)=\min\left\lbrace \Phi(\zeta(t,x))+p_1(\zeta(t,x))\delta e^{-\mu t},  1\right\rbrace,\\
		\underline {v}(t, x)=\max\left\lbrace\Psi(\zeta(t,x))-p_2(\zeta(t,x))\delta e^{-\mu t}, 0\right\rbrace,\\
	\end{cases}
\end{equation}
where 
$$\zeta(t, x)=-h_{\varepsilon}(|x|)-(c_{uv}+\frac{\varepsilon}{2})t
+\omega\delta e^{-\mu t}-\omega\delta +R-M.$$

We will show that $(\overline{u},\underline {v})$ is a supersolution.

We first check the initial values. For $t=0$ and $|x|<R$, one has that 
$$\overline u(0, x)\ge \delta\ge \overline u_{R}(0, x), ~
\underline v(0, x)\le 1-\delta\le \underline v_{R}(0, x).$$
For $t=0$ and $|x|\ge R$, one has that $\zeta(0, x)=-h_{\varepsilon}(|x|)
+R-M\le -M$ and 
$$
\overline u(0, x)\ge \min\left\lbrace 1-\delta+\delta, 1\right\rbrace\ge 
 \overline u_{R}(0, x), ~\underline v(0, x)\le \max\left\lbrace
 \delta -\delta, 0 \right\rbrace\le  \underline v_{R}(0, x).
$$

Then we check that 
\[N_{1}[\overline u, \underline v]:=\overline u_{t}-d\Delta \overline u
-r\overline u(1-\overline u-a\underline v)\ge0,
\]
for $0\le t\le \frac{R-R_{\varepsilon}}{c_{uv}+\varepsilon}$ and $x\in \mathbb{R}^{N}$ such that $\overline u<1$ and 
\[N_{2}[\overline u, \underline v]:=\underline v_{t}-\Delta \underline v
-\underline v(1-\underline v-b\overline u)\le 0,
\]
for $0\le t\le \frac{R-R_{\varepsilon}}{c_{uv}+\varepsilon}$ and $x\in \mathbb{R}^{N}$ such that $\underline v>0$.

After some computation and by (\ref{pro2}), we obtain
\begin{equation}\label{n3}
	\begin{aligned}
		N_{1}[\overline u, \underline v]
		=&-\left( \frac{\varepsilon}{2} +\omega\delta\mu e^{-\mu t}\right)  \Phi' + d(1-|h'_{\varepsilon}|^2) \Phi''
		+d\left(h^{''}_{\varepsilon}+\frac{N-1}{|x|}
		h^{'}_{\varepsilon} \right)\Phi' - p_1 \delta \mu e^{-\mu t}\\
		&-(c_{uv}+\frac{\varepsilon}{2} +\omega\delta\mu e^{-\mu t}) p_1'\delta e^{-\mu t} -d|h'_{\varepsilon}|^2 p_1''\delta e^{-\mu t}
		+d\left(h^{''}_{\varepsilon}+\frac{N-1}{|x|}
		h^{'}_{\varepsilon} \right)p_1'\delta e^{-\mu t} \\
		&+r\Phi(1-\Phi-a\Psi)-r\overline{u}(1-\overline{u}-a\underline{v}).
	\end{aligned}
\end{equation}
For $\zeta(t, x)\le -M$, one still has that $p_{1}\equiv 2a$, $p_{2}\equiv 1$, 
$\Phi(\zeta(t, x))\ge 1-\delta$, $\Psi(\zeta(t, x))\le \delta$. Therefore, 
$\overline u(t, x)-\Phi(\zeta(t, x))=
2a\delta e^{-\mu t}$, $\underline v(t, x)-\Psi(\zeta(t, x))=-\delta e^{-\mu t}$, 
$\overline u(t, x)\ge \Phi(\zeta(t, x))\ge 1-\delta$. Then
\begin{align*}
	r\Phi(1-\Phi-a\Psi)-r\overline{u}(1-\overline{u}-a\underline{v})
	=& r\Phi(\overline{u}-\Phi+a\underline{v}-a\Psi) -2ar\delta e^{-\mu t}(1-\overline{u}-a\underline{v})\\
	\ge& r (1-\delta)a\delta e^{-\mu t}-\delta2ar\delta e^{-\mu t}\\
	\ge& \frac{1}{2}ar\delta e^{-\mu t},
\end{align*}
by $\delta\le \delta _0$. Notice that 
the inequalities $\zeta(t, x)\le -M$ and $0\le t\le \frac
{R-R_\varepsilon}{c_{uv}+\varepsilon}$ yield
$$
h_\varepsilon(|x|)\ge -(c_{uv}+\frac{\varepsilon}{2})t-\omega\delta +R\ge 
H_\varepsilon,
$$ 
by (\ref{Rcon}). Therefore $h'_\varepsilon(|x|)=1$ and 
$h''_\varepsilon(|x|)=0$. Combined with \eqref{hcon}, \eqref{mu1}, $\Phi'<0$ and $\Phi''<0$ for $\xi\le -M$, it follows from \eqref{n3} that $N_{1}[\overline u, \underline v]\ge 0$ for $\xi\le -M$. 

For $\zeta(t, x)\ge M$, one also has that $p_{1}\equiv 1$, $p_{2}\equiv
2b$, $\Phi(\zeta(t, x))\le \delta$, $\Psi(\zeta(t, x))\ge 1-\delta$. 
Therefore, $\overline u(t, x)-\Phi(\zeta(t, x))=\delta e^{-\mu t}$, $\underline  v(t, x)-\Psi(\zeta(t, x))=
-2b\delta e^{-\mu t}$ and $\underline v(t, x)\ge 1-\delta -2b\delta$. Then 
\begin{align*}
	r\Phi(1-\Phi-a\Psi)-r\overline{u}(1-\overline{u}-a\underline{v})
	=& r\Phi(\overline{u}-\Phi+a\underline{v}-a\Psi) -r\delta e^{-\mu t}(1-\overline{u}-a\underline{v})\\
	\ge& -r\delta 2ab\delta e^{-\mu t}- r\delta e^{-\mu t}(1-a(1-\delta -2b\delta))\\
	\ge& \frac{a-1}{2}r\delta e^{-\mu t},
\end{align*}
by $\delta \le \delta_0$. Combined with \eqref{hcon}, \eqref{mu1}, $\Phi'<0$ and $\Phi''>0$ for $\xi\ge M$, it follows from \eqref{n3} that $N_{1}[\overline u, \underline v]\ge 0$ for $\xi\ge M$.

For $-M\le \zeta(t, x)\le M$, one has that $-\Phi'(\zeta(t, x))\ge k_{1}$. Notice that $\underline v(t, x)-\Psi(\zeta(t, x))\ge -2b\delta e^{-\mu t}$. 
Then, 
\begin{align*}
	r\Phi(1-\Phi-a\Psi)-r\overline{u}(1-\overline{u}-a\underline{v})
	=& r\Phi(\overline{u}-\Phi+a\underline{v}-a\Psi) -p_{1}r\delta e^{-\mu t}(1-\overline{u}-a\underline{v})\\
	\ge& - 2abr\delta e^{-\mu t}- 2a\delta e^{-\mu t}\\
	\ge& -2a(b+1)r\delta e^{-\mu t}.
\end{align*}
From  $\zeta(t, x)\le M$ and $0\le t\le \frac
{R-R_\varepsilon}{c_{uv}+\varepsilon}$ , one also obtain
$h_\varepsilon(|x|)\ge-(c_{uv}+\frac{\varepsilon}{2})t-w\delta +R-M-M_\varepsilon\ge
H_\varepsilon $ by \eqref{Rcon}. Therefore, $h'_\varepsilon(|x|)=1$ and $h''_\varepsilon(|x|)=0$. Combined with \eqref{hcon}, \eqref{p},
\eqref{omega}  and \eqref{mu1}, it follows from \eqref{n3} that
\begin{align*}
	N_1[\overline u, \underline v]\ge &k_{1}\omega\delta \mu e^{-\mu t}
	-2a\delta \mu e^{-\mu t}-2a(b+1)r\delta \mu e^{-\mu t}\\
	&-\|p'_1\|_{L^{\infty}}\delta \omega\delta e^{-\mu t}-(c_{uv}\|p'_1\|_{L^{\infty}}+d\|p''_1\|_{L^{\infty}}+
	\frac{\varepsilon}{2}d\|p'_1\|_{L^{\infty}})\delta \mu e^{-\mu t}\\
	\ge &\frac{1}{2}\omega k_{1}\delta \mu e^{-\mu t}-2a\delta \mu e^{-\mu t}-2a(b+1)r\delta \mu e^{-\mu t}-\delta \mu e^{-\mu t}\ge 0.
\end{align*} 
As a conclusion, $N_1[\overline u, \underline v]\ge 0$ for $0\le t\le \frac{R-R_{\varepsilon}}{c_{uv}+\varepsilon}$ and $x\in \mathbb{R}^{N}$.

Next one can apply the similar proof used above to show that $N_2
[\overline u, \underline v]\le 0$. By (\ref{pro2}) we obtain
\begin{equation}\label{n4}
	\begin{aligned}
		N_2[\overline u, \underline v]=&-\left( \frac{\varepsilon}{2}+\omega\delta \mu e^{-\mu t}\right) \Psi'+
		(1-|h'_\varepsilon|^2)\Psi''+\left( h''_\varepsilon+
		\frac{N-1}{|x|}h'_\varepsilon\right)\Psi'+p_{2}\delta 
		\mu e^{-\mu t} \\
		&+\left(c_{uv}+ \frac{\varepsilon}{2}+\omega\delta \mu e^{\mu t}\right)p'_2\delta e^{-\mu t}+\delta e^{-\mu t}p''_{2} 
		\delta e^{-\mu t} +\left( h''_\varepsilon+
		\frac{N-1}{|x|}h'_\varepsilon\right)p'_2\delta e^{-\mu t}\\
		&+\Psi(1-\Psi-b\Phi)-\underline v(1-\underline v-b\overline u).
	\end{aligned}
\end{equation}
If $\zeta(t, x)\le -M$, one has that $p_1\equiv 2a$, $p_2\equiv 1$, $\Phi(\zeta(t, x))\ge 1-\delta$, $\Psi(\zeta(t, x))\le \delta$. $\overline u(t, x)-\Phi(\zeta(t, x))=
2a\delta e^{-\mu t}$, $\underline v(t, x)-\Psi(\zeta(t, x))=-\delta e^{-\mu t}$,  
$\overline u(t, x)\ge \Phi(\zeta(t, x))\ge 1-\delta$. Then
\begin{align*}
	\Psi(1-\Psi-b\Phi)-\underline{v}(1-\underline{v}-b\overline{u})
	=& \Psi(\underline{v}-\Psi+b\overline{u}-b\Phi) +\delta e^{-\mu t}(1-\underline{v}-b\overline{u})\\
	\le& \delta 2ab\delta e^{-\mu t}+\delta e^{-\mu t}(1-b+b\delta)\\
	\le& -\frac{b-1}{2}\delta e^{-\mu t},
\end{align*}
by $\delta\le \delta_0$. Recall that $\zeta(t, x)\le -M$ implies 
$h_\varepsilon(|x|)\ge H_\varepsilon$, hence $h'_\varepsilon(|x|)=1$ and 
$h''_\varepsilon(|x|)=0$. Combined with \eqref{hcon}, \eqref{mu1}, $\Phi'<0$ and $\Psi''>0$ for $\xi\le -M$, it follows from \eqref{n4} that $N_{2}[\overline u, \underline v]\le 0$ for $\xi\le -M$.
 
If $\zeta(t, x)\ge M$, one has that $p_{1}\equiv 1$, $p_{2}\equiv
2b$, $\Phi(\zeta(t, x))\le \delta$,$ \Psi(\zeta(t, x))\ge 1-\delta$. 
Therefore, $\overline u(t, x)-\Phi(\zeta(t, x))=\delta e^{-\mu t}$, $\underline  v(t, x)-\Psi(\zeta(t, x))=
-2b\delta e^{-\mu t}$ and $\underline v(t, x)\ge 1-\delta -2b\delta$. Then
\begin{align*}
	\Psi(1-\Psi-b\Phi)-\underline{v}(1-\underline{v}-b\overline{u})
	=& \Psi(\underline{v}-\Psi+b\overline{u}-b\Phi) +2b\delta e^{-\mu t}(1-\underline{v}-b\overline{u})\\
	\le& -(1-\delta) b\delta e^{-\mu t}+2b\delta e^{-\mu t}(1+2b)\delta\\
	\le& -\frac{1}{2}b\delta e^{-\mu t},
\end{align*} 
by $\delta\le \delta_0$. Combined with \eqref{hcon}, \eqref{mu1}, $\Psi'>0$ and $\Psi''<0$ for $\xi\ge M$, it follows from \eqref{n4} that $N_{2}[\overline u, \underline v]\le 0$ for $\xi\ge M$.

If $-M\le \zeta(t, x)\le M$, one has that $\Psi'(\zeta(t, x))\ge k_{2}$.
Notice that $\overline u(t, x)-\Phi(\zeta(t, x))\le -2a\delta e^{-\mu t}$. 
Then,
\begin{align*}
	\Psi(1-\Psi-b\Phi)-\underline{v}(1-\underline{v}-b\overline{u})
	=& \Psi(\underline{v}-\Psi+b\overline{u}-b\Phi) +p_2\delta e^{-\mu t}(1-\underline{v}-b\overline{u})\\
	\le&  2ab\delta e^{-\mu t}+2b\delta e^{-\mu t}\\
	\le& 2b(a+1)\delta e^{-\mu t},
\end{align*}
Recall that $\zeta(t, x)\le M$ implies 
$h_\varepsilon(|x|)\ge H_\varepsilon$ and therefore
 $h'_\varepsilon(|x|)=1$ and $h''_\varepsilon(|x|)=0$. 
 Combined with \eqref{hcon}, \eqref{p},  \eqref{omega}
 and \eqref{mu1}, it follows from \eqref{n4} that
  \begin{align*}
  	N_2[\overline u, \underline v]\le &-k_{2}\omega\delta \mu e^{-\mu t}
  	+2b\delta \mu e^{-\mu t}+2b(a+1)r\delta \mu e^{-\mu t}\\
  	&+\|p'_2\|_{L^{\infty}}\delta \omega\delta e^{-\mu t}+(c_{uv}\|p'_2\|_{L^{\infty}}+\|p''_2\|_{L^{\infty}}+
  	\frac{\varepsilon}{2}\|p'_2\|_{L^{\infty}})\delta \mu e^{-\mu t}\\
  	\le &-\frac{1}{2}\omega k_2\delta \mu e^{-\mu t}+2b\delta \mu e^{-\mu t}+2b(a+1)r\delta \mu e^{-\mu t}+\delta \mu e^{-\mu t}\le 0.
  \end{align*} 
\end{proof}
As a conclusion, $N_2[\overline u, \underline v]\le 0$ for $0\le t\le \frac{R-R_{\varepsilon}}{c_{uv}+\varepsilon}$ and $x\in \mathbb{R}^{N}$.
 
{\it \textbf{Step 3}: proof of \eqref{delta1} and \eqref{delta2}.} Therefore by the comparison principle, one has that
 \[
 \overline u_R(t, x)\le \overline u(t, x)~\text{and}~\underline v_R(t, x)
 \ge \underline v(t, x)~\text{for}~0\le t\le \frac{R-R_{\varepsilon}}{c_{uv}+\varepsilon}~\text{and}~x\in \mathbb{R}^{N}.
 \] 
Then since \eqref{hcon2}, $0\le t\le \frac{R-R_{\varepsilon}}{c_{uv}+\varepsilon}$
 and $|x|\le R-R_{\varepsilon}-(c_{uv}+\varepsilon)t$ yield 
 \[
 \zeta(t, x)\ge -h_\varepsilon(0)-|x|-\left(c_{uv}
 +\frac{\varepsilon}{2} \right)t-\omega\delta +R-M
 \ge M_\varepsilon\ge M.
 \]
Then it from (\ref{eqsup}) that
\[
\overline u_R(t, x)\le \min\{\delta+\delta, 1\}=2\delta, ~
\hbox{and}~\underline v_R(t, x)\ge \max\{1-\delta-\delta, 0\}=1-
2\delta.
\]
Thus (\ref{delta1}) is concluded.

Furthermore, there exists $T_\varepsilon>0$ such that $\delta e^{-\mu t}\le 
\delta/2\le\delta_0/2$ for $T_\varepsilon\le t\le \frac{R-R_{\varepsilon}}{c_{uv}+\varepsilon}$. Then (\ref{delta2})
immediately follows from (\ref{eqsup}).

\section{Proofs for scenario \textbf{(C1)}.}\label{sec5}
In this section, the initial values are always assumed to satisfy \textbf{(C1)} and let $(u, v)$ be the related solution.	
 We will prove Theorem  \ref{thm2} and Theorem \ref{thm3} in this section.

We first show   that both  the species 
$u$ and $v$ can not spread faster than 
$c_{u}$ and $c_{v}$ respectively. 

\begin{lemma}\label{lemma:(0,0)}
	There holds
	\[
	\lim_{t\to +\infty}\sup_{|x|\ge ct}u(t, x)= 0, \quad \hbox{ for any $c>c_{u}$},
	\]
	and
	\[
	\lim_{t\to +\infty}\sup_{|x|\ge ct} v(t, x)= 0, \quad \hbox{ for any $c>c_{v}$}.
	\]
\end{lemma}

\begin{proof}
Recall that $\phi(x\cdot e-c_ut)$  denotes the traveling front satisfying
	\begin{equation}
		\begin{cases}
			d\phi''+c_u\phi'+
			r \phi(1-\phi)=0,
			\\
			\phi(-\infty)=1,~ \phi(+\infty)=0,
		\end{cases}
	\end{equation}
	where $c_{u}=2\sqrt{dr}>0$ and $e\in \mathbb{S}^{N-1}$.
	
	Define  $\overline u(t, x)=\min\{Xe^{-\lambda_u(x\cdot e-
		c_{u}t)}, 1\}$ with $\lambda_u=\sqrt{r/d}$ for 
$e\in \mathbb{S}^{N-1}$ and $X>0$, $\underline v(t, x)= 0$.
 By some computation, one can get
	\[
	\begin{aligned}
		N_{1}[\overline u, \underline v]&=\overline u_{t}-
		d\Delta \overline u-r\overline u(1-\overline u-a\underline v)\\
		&=(-d\lambda_u^{2}+c_{u}\lambda_u-r)\overline u+r\overline u^{2}
	  =r\overline u^{2}	\ge 0,
	\end{aligned}
	\]
	and $N_{2}[\overline u, \underline v]=\underline v_{t}-
	\Delta \underline v-\underline v(1-\underline v-b\overline u)=0$.
	
	 Actually, one can choose $X$ large enough such that  $u(0, x)$ satisfies
	$u(0, x)\le Xe^{-\lambda_u(x\cdot e-
		c_{u}t)}$ for all $x\in \mathbb{R}^{N}$ and every $e\in \mathbb{S}^{N-1}$ since the support of $u_0$ is bounded by \textbf{(C1)}.
	Together with $v(0, x)\ge 0$, it follows from the comparison principle that $u(t, x)\le \overline u(t, x)$, $v(t, x)\ge  \underline v(t, x)$ for 
	$t\ge 0$ and $x\in \mathbb{R}^{N}$.
	Thus for any $c>c_{u}$ and $e\in \mathbb{S}^{N-1}$,
	there holds
	\[
	\lim_{t\to +\infty}\sup_{|x|\ge ct}u(t, x)\le 
	\lim_{t\to +\infty}\sup_{|x|\ge
		ct}\inf_{e\in\mathbb{S}^{N-1}}Xe^{-\lambda_u(x\cdot e-
		c_{u}t)}\le \lim_{t\to +\infty}
	Xe^{-\lambda_u(ct-
		c_{u}t)} \to 0,
	\]
	
Similarly, one can define $\underline u(t, x)=0 $, $\overline v (t, x)=\min\{ Ye^{-\lambda_v(x \cdot e-c_vt)}, 1\}$ with $\lambda_v=1$ and 
$Y>0$, where $c_v=2, e \in  S^{N-1}$.
The parameter $Y$ is also taken sufficiently large such that the comparison between the initial values works. Then we can apply once again the comparison principle
to conclude that
\[
\lim_{t\to +\infty}\sup_{|x|\ge ct}v(t, x)\to 0,
\]
for $c>c_v$.
\end{proof}

Then, for $c_{v}>c_{u}$, we have shown that $u$ converges to $0$ for $|x|\ge ct$ with $c>c_{u}$  as $t\rightarrow +\infty$ and then we show that $v$  occupies the intermediate zone $\{x\in\mathbb{R}^N:\, c_{1}t\le |x|\le c_{2}t\}$ for $c_{u}<c_{1}\le c_{2}<c_{v}$.  The strategy is to show the convergence of $v$ to 1 along $cte$ for $c \in (c_{u}, c_{v})$ and any $e\in\mathbb{S}^{N-1}$ and then show the uniform convergence for $c_{1}t\le |x|\le c_{2}t$. Before doing so, let us introduce the following equation for $v(t, x)$:
\begin{equation}\label{veq1}
	v_{t}=\Delta v+v(1-v-b\varepsilon),
\end{equation}
with $\varepsilon>0$. Obviously the reaction term is Fisher-KPP type and the equation has 
two equilibrium  $0$, $1-b\varepsilon$. By classical results of Aronson and Weinberger \cite{Aro} and \cite[Lemma~4.2]{DGM}, we have the following lemma.

\begin{lemma}\label{vlemm2}
	For any small $\varepsilon>0$ ($<1/b$), denote $\beta_{\varepsilon}=
	1-b\varepsilon$. There is a minimal speed
	$c_{\varepsilon}=2\sqrt{1-b\varepsilon}>0$ such that for every $e\in\mathbb{S}^{N-1}$, (\ref{veq1}) admits a traveling front $\psi_{\varepsilon}(x\cdot e -ct)$ connecting $0$ to $\beta_{\varepsilon}$ if and only if	$c \ge  c_{\varepsilon}$.
	 For any nontrivial, compactly supported and
		continuous initial data $0 \le  v_{0}(x)\le \beta_{\varepsilon}$,
		the associated solution $v(t, x)$ of (\ref{veq1}) has a spreading  speed $c_{\varepsilon}$ in the following sense:
		\[
		\begin{cases}
		\lim_{t\to +\infty}\sup_{|x|\ge ct} v(t, x)=0, ~& \hbox{for }\forall 
		c>c_{\varepsilon}\\
		\lim_{t\to +\infty}\sup_{|x|\le ct} |\beta_{\varepsilon}
		-v(t, x)|=0, ~&\hbox{for }\forall c\in (0, c_{\varepsilon}).
		\end{cases}
		\]
	Furthermore, $c_{\varepsilon}$ and $\beta_{\varepsilon}$ are
	nonincreasing with respect to $\varepsilon$ and tend respectively to
	$ c_{v}$ and 1 as $\varepsilon \to  0$.
\end{lemma}

\begin{lemma}\label{lemm1}
	For any given $c \in (c_{u}, c_{v})$, the solution $(u, v)$ of (\ref{pro1}) satisfies
	\begin{equation}\label{vc1}
		\lim_{t\to +\infty} v(t, x+cte)=1,
	\end{equation}
	where the convergence holds locally uniformly with respect to $x$ and uniformly
	with respect to $e\in \mathbb{S}^{N-1}$.
\end{lemma}
\begin{proof}
Since the proof is similar as that of	\cite[Lemma 4.3]{DGM},  we only sketch	it.

\begin{itemize}
\item  By the proof of Lemma~\ref{lemma:(0,0)}, one has that for any small $\varepsilon>0$, there exists some $X_{\varepsilon}>0$ such that 
	\begin{equation*}
		u(t, x)\le \overline{u}(t, x)\le \varepsilon ~\text{for any}~(t, x)~
		\text{satisfies}~
		|x|\ge X_{\varepsilon}+c_{u}t,
	\end{equation*}
	where $\overline{u}(t, x)=X_0 e
	^{-\lambda_u(x\cdot e -c_u t)}$ by setting $X_{0}>0$ such that 
	$u(t, x)\le \overline{u}(t, x)$ for 
	$t>0$ and $x\in \mathbb{R}^N$.
	
	\item Take $c'$ such that $c_{u}<c<c'<c_{v}$. Let $\psi_R$ be the principal eigenfunction normalized by $||\psi_{R}||_{\infty}=1$ and $\lambda_R$ be the principal eigenvalue of the following problem
	\begin{equation*}
		\begin{cases}
			\Delta \psi_{R}=\lambda_{R}\psi_{R}&~\text{in}~B_{R},\\
			\psi_{R}>0&~\text{in}~B_{R},\\
			\psi_{R}=0 &~\text{on}~\partial B_{R}.
		\end{cases}
	\end{equation*}
	 For every $e\in\mathbb{S}^{N-1}$, there is a subsolution of \eqref{veq1} as
	\begin{equation*}
		\underline{v}(t, x)=\kappa e^{-\frac{c'}{2}(x\cdot e-c' t)}\psi_{2R}
		(x-(c' t+X_{\varepsilon}+2R)e).
	\end{equation*}
by taking $R>0$ sufficiently large, satisfying $\underline{v}_t -\Delta \underline{v} -\underline{v}(1-\underline{v}-bu)\le 0$ for $t\ge 0$ and $|x|\ge X_{\varepsilon} +c_u t $. The comparison principle leads to the fact that there are
	$R_{1}>0$ and $\kappa_{1}>0$ such that 
	\begin{equation}\label{cla1}
		\liminf_{t\to+\infty}\inf_{e\in \mathbb{S}^{N-1}}\inf_{x\in B_{R_{1}}}
		v\left( \frac{ct}{c'}, x+(ct+X_{\varepsilon}+2R_{1})e\right)
		>\kappa_{1}. 
	\end{equation}

\item By noticing that $\eta' \psi_{R/2}(x)$ is a stationary subsoulution of \eqref{veq1} for sufficiently small $\eta'>0$, it follows from \eqref{cla1} that
\begin{equation}\label{vcon1}
		\liminf_{t\to+\infty}\inf_{e\in \mathbb{S}^{N-1}}\inf_{(\tau, x)\in Q_{t}}
		v\left( \tau, x+(ct+X_{\varepsilon}+2R_{2})e\right)>\kappa_{2},
	\end{equation}
	for some $R_{2}>0$ and $\kappa_{2}>0$,  where $Q_{t}$ is defined by
	\[
	Q_{t}:=\left\lbrace (\tau, x) ~| ~\frac{ct}{c^{'}}\le \tau\le t~\text{and}~x\in
	\mathbb{R}^{N}\right\rbrace. 
	\]
	
\item Extract a subsequence $t_{n}\to +\infty$ and
	$ e_{n}\to e_{\infty}$ in $\mathbb{S}^{N-1}$ such that
	\[
	v_{n}(t, x)=v(t+t_{n}, x+ct_{n}e)
	\]
	converges locally uniformly to $v_{\infty}(t, x)$ satisfying
	\begin{equation}\label{eq:vinfty}
	\partial_{t}v_{\infty}-\Delta v_{\infty}-v_{\infty}(1-v_{\infty}-b\varepsilon)\ge 0 \quad	~\text{for}~(t, x)\in \mathbb{R}\times \mathbb{R}^{N}.
	\end{equation}
	By \eqref{vcon1}, we know that
	\[
	\inf_{x\in B_{R_{2}}}v_{\infty}(t, x+(X_{\varepsilon}+2R)e_{\infty})\ge \kappa_{2}>0
	\] 
	for all $t\le 0$. Then, by \eqref{eq:vinfty} and the hair-trigger effect for the Fisher-KPP equation (see \cite{Aro}), one has $v_{\infty}(t,x)\ge 1-b\varepsilon$ as $t\rightarrow +\infty$ locally uniformly for $x\in\mathbb{R}^N$. The arbitrariness of $\varepsilon$ leads the conclusion \eqref{vc1}.
\end{itemize}
\end{proof}
 
Then, we introduce another tool in the paper of Aronson and Weinberger \cite{Aro}.

\begin{lemma}\label{lemm3}
	For any $0<\varepsilon<1/b$ and  any $c\in (0, c_{\varepsilon})$
	($c_{\varepsilon}=2\sqrt{1-b\varepsilon}$), there exist 
	$\beta(c)<1-b\varepsilon$ such that for any $\beta\ge \beta(c)$
	the solution of 
	\begin{equation*}
		\begin{cases}
			\hat{V}''+c\hat{V}'+\hat{V}(1-\hat{V}-b\varepsilon)=0,\\
			\hat{V}(0)=\beta, ~~\hat{V}'(0)=0
		\end{cases}
	\end{equation*}
	satisfies $\hat{V}(a)=0$ for some $a>0$ and $\hat{V}'(\xi)<0$ for
	$\xi \in (0, a]$.
\end{lemma}

\begin{lemma}\label{lemm2}
	Assume $c_{u}<c_{v}$. For any given $c_{1}$, $c_{2}$ satisfying $c_{u}<c_{1}<c_{2}<c_{v}$,
	 the 
	solution $(u,v)$ has 
	the following spreading property
	\begin{equation*}
		\lim_{t\to +\infty}\sup_{c_{1}t\le |x|\le c_{2}t}\left\lbrace |u(t, x)|
		+|v(t, x)-1|\right\rbrace =0.
	\end{equation*}
\end{lemma}

\begin{proof}
	We still sketch the proof by referring to that of \cite[Theorem~2.1]{DGM}. 
\begin{itemize}
\item Take small $\varepsilon>0$ such that $c_{\varepsilon}\in (c_2,c_{v})$ and take  $c\in (c_{2}, c_{\varepsilon})$. Define two family 
	of functions 
	\begin{equation*}
		\underline{v}_{i}(t, x; e) = 
		\begin{cases}
			\beta & \text{if } |x-c_{i}te|<\rho,\\
			\hat{V}(|x-c_{i}te|-\rho) & \text{if } \rho\le 
			|x-c_{i}te|\le \rho+a,\\
			0 & \text{if } |x-c_{i}te| > \rho+a,
		\end{cases} 
	\end{equation*}
	where $i=1, 2$, $e\in \mathbb{S}^{N-1}$ and $\beta\ge \beta(c)$ is
	defined by Lemma \ref{lemm3} and $\rho>0$ is sufficiently large. Lemma \ref{lemm1} implies that there exists $T>0$ such
	that $v(T, x)> \beta$
	for all $cT-\rho-a\le |x|\le cT+\rho+a$. Then defining $T_{i}=\frac
	{c}{c_{i}}T> T$, one gets
	$v(T, x)\ge \underline v_i(T_{i}, x; e)$ due to 
	the definition of $\underline v_i(t_{i}, x, e)$
	for $i=1, 2$ and $e\in \mathbb{S}^{N-1}$. 
	
	\item One can check that $\underline{v}_{i}(t, x; e)$ are subsolutions of \eqref{veq1} for $t\ge 0$ and $x\in\mathbb{R}^N$. It follows from the comparison principle that
	\begin{equation*}
		v(t+T, x)\ge \underline {v}_i(t+T_{i}, x; e).
	\end{equation*}
	Thus, $v(t,x)\ge \beta$ for $t\ge T$ and $x\in   \partial \Pi_{t}$ where
	$$\Pi_{t}:=\left\lbrace x\in \mathbb{R}^{N}|~
	cT + c_{1}(t - T ) \le |x|\le cT + c_{2}(t - T )\right\rbrace.$$
	Then, the maximum principle implies that 
	\[
	\liminf_{t\to+\infty}\inf_{c_{1}t+(c-c_{1})T \le |x|\le c_{2}t+
		(c-c_{2})T} v(t, x)\ge \beta.
	\]
	
	\item By slightly 
	modifying $c_{1}$ and $c_{2}$, one has
	$$
	\liminf_{t\to+\infty}\inf_{c_{1}t\le |x|\le c_{2}t} v(t, x)\ge \beta.
	$$
	Since $\varepsilon$ can be chosen  arbitrarily small and $\beta$ can 
	arbitrarily close to $1-b\varepsilon$, it derives
	\[
	\liminf_{t\to+\infty}\inf_{c_{1}t\le |x|\le c_{2}t} v(t, x)\ge 1.
	\]
\end{itemize}	
\end{proof}

For the case $c_{u}>c_{v}$,  we only have to switch the roles of $ u$ and $v$ and adapt the same proofs.
 
\begin{corollary}\label{cu2}
Assume $c_{u}>c_{v}$. For any given $c_{1}, c_{2}$ satisfying $c_{v}<c_{1}<c_{2}<c_{u}$,
	 the 	solution $(u,v)$ has 
	the following spreading property
	\begin{equation}
		\lim_{t\to +\infty}\sup_{c_{1}t\le |x|\le c_{2}t}\left\lbrace |u(t, x)-1|
		+|v(t, x)|\right\rbrace =0.
	\end{equation}
\end{corollary}

We are now able to complete the proofs of Theorems \ref{thm2} and \ref{thm3}.

\begin{proof}[Proof of Theorem \ref{thm2}.]
Since \textbf{(A3)} is satisfied, there is $T>0$ such that $u(T,x)\ge 1-\delta$ and $v(T,x)\le \delta$ for $x\in B(0,\rho)$ where $\delta$ and $\rho$ are given by Lemma~\ref{lemma:sub}. By the comparison principle and Lemma~\ref{lemma:sub}, one has that
\begin{equation}\label{limcuv}
\lim_{t\to +\infty}\sup_{|x|\le ct} |u(t, x)-1| =0,\, \lim_{t\to +\infty}\sup_{|x|\le ct} |v(t, x)| =0\quad \text{for}~ 0<c<c_{uv}.
\end{equation}
We now show that for any $\varepsilon\in (0,c_u)$,
\begin{equation}\label{convergence-cu}
(u(t,x),v(t,x))\rightarrow (1,0)~\text{
		uniformly in}~\left\lbrace x\in \mathbb{R}^{N}; |x|\le 
	(c_{u}-\varepsilon)t\right\rbrace ~\text{as}~t\to +\infty.
	\end{equation}

Take $0<\varepsilon<\min\{c_{uv},c_u,c_v,(c_u-c_v)/2\}$. By Corollary~\ref{cu2}, there is $T_1>0$ such that 
$$\frac{T_1(c_u-c_v-2\varepsilon)}{2}> \rho,$$
and
\begin{equation}\label{T1}
u(t,x)\ge 1-\delta,\, v(t,x)\le \delta, \hbox{ for any $t\ge T_1$ and  $(c_v+\varepsilon)t\le |x|\le (c_u-\varepsilon)t$}.
\end{equation}
Then, for any $e\in\mathbb{R}^N$ and any $x_0\in \{x\in\mathbb{R}^N:\,(c_v+\varepsilon)T_1+\rho \le |x|\le (c_u-\varepsilon)T_1-\rho\}$, one has
$$u(T_1+t,x)\ge \underline{u}_{\rho}(t,x-x_0), \, v(T_1+t,x)\le \overline{v}_{\rho}(t,x-x_0),\quad \hbox{for $t\ge 0$ and $x\in\mathbb{R}^N$}.$$
Let 
$$\tau:=\frac{(c_u-c_v-2\varepsilon)T_1 -2\rho}{c_{v} +\varepsilon}.$$
By  Lemma~\ref{lemma:sub}, one has that $u(T_1+\tau,x)\ge 1-\delta$, $v(T_1+\tau,x)\le \delta$,  for  $(c_v+\varepsilon)T_1+\rho \le |x|\le (c_u-\varepsilon)T_1-\rho$. Moreover, by \eqref{T1}, $u(T_1+\tau,x)\ge 1-\delta$, $v(T_1+\tau,x)\le \delta$,  for  $(c_u-\varepsilon)T_1-2\rho \le |x|\le (c_u-\varepsilon)(T_1+\tau)$. That is,
\begin{equation*} 
u(t,x)\ge 1-\delta,\, v(t,x)\le \delta, \hbox{ for  $ T_1\le t\le T_1+\tau$ and  $(c_v+\varepsilon)T_1\le |x|\le (c_u-\varepsilon)t$}.
\end{equation*}

By induction, one can show that
\begin{equation*} 
u(t,x)\ge 1-\delta,\, v(t,x)\le \delta, \hbox{ for  $ T_1\le t\le T_1+n\tau$ and  $(c_v+\varepsilon)T_1\le |x|\le (c_u-\varepsilon)t$},
\end{equation*}
for any $n\in\mathbb{N}$. Combined with \eqref{limcuv}, it follows that
\begin{equation*} 
u(t,x)\ge 1-\delta,\, v(t,x)\le \delta, \hbox{ for all $ t\ge T_1$ and  $|x|\le (c_u-\varepsilon)t$}.
\end{equation*}
By the arbitrariness of $\delta_0$, one gets \eqref{convergence-cu}.

Then, the conclusion of Theorem \ref{thm2} follows from \eqref{convergence-cu} and Lemma~\ref{lemma:(0,0)}.
\end{proof}

\begin{proof}[Proof of Theorem~\ref{thm3}]
 Similar as Theorem~\ref{thm2}, one can easily show \eqref{limcuv}. Then, by Lemma~\ref{lemma:(0,0)}, we only have to show that
\begin{equation}\label{convergence-cuv-cv}
\lim_{t\to +\infty}\sup_{c_{1}t\le|x|\le c_{2}t}
\left\lbrace |u(t, x)| +|v(t, x)-1|\right\rbrace 
=0,~\text{for}~
c_{uv}<c_{1}\le c_{2}<c_{v}.
\end{equation}

Take $0<\varepsilon<\min\{c_{uv},c_u,c_v,(c_u-c_v)/2\}$ small enough. Let  $R_{\varepsilon}$ and $T_{\varepsilon}$ be given by Lemma~\ref{lem1}. By Lemma~\ref{lemm2}, there is $T>0$ such that 
$$(c_v-c_u-2\varepsilon)T/2\ge R_{\varepsilon},\quad \tau:=\frac{(c_v-c_u-2\varepsilon)T-R_{\varepsilon}}{c_u+c_{uv}+2\varepsilon}\ge \max\left\{T_{\varepsilon},\frac{R_{\varepsilon}}{\varepsilon}\right\},$$
and
\begin{equation}\label{trapcucv}
u(t,x)\le \delta,\, v(t,x)\ge 1-\delta, \hbox{ for any $t\ge T$ and  $(c_u+\varepsilon)t\le |x|\le (c_v-\varepsilon)t$}.
\end{equation}
Let $R:= (c_v-c_u-2\varepsilon)T/2\ge R_{\varepsilon}$.
Then, for any $e\in\mathbb{R}^N$ and any $x_0\in \mathbb{R}^N$ such that $ |x_0|= (c_v+c_u)T/2$, notice that $(c_u+\varepsilon)T\le |x|\le (c_v-\varepsilon)T$ for $x\in B(x_0,R)$ and one has
$$u(T+t,x)\le \overline{u}_{R}(t,x-x_0), \, v(T+t,x)\ge \underline{v}_{R}(t,x-x_0),\quad \hbox{for $t\ge 0$ and $x\in\mathbb{R}^N$}.$$
By Lemma~\ref{lem1}, one has that $u(T+t,x)\le 2\delta$, $v(T+t,x)\ge 1-2\delta$ for $0\le t\le \tau$ and $|x-x_0|\le R-R_{\varepsilon}-(c_{uv}+\varepsilon) t$. In particular, $u(T+\tau,x)\le \delta$, $v(T+\tau,x)\ge 1-\delta$, for $|x-x_0|\le R-R_{\varepsilon}-(c_{uv}+\varepsilon) \tau$.
Since $x_0$ could be any point such that $|x_0|=(c_v+c_u)T/2$, one further gets that
$u(T+t,x)\le 2\delta$, $v(T+t,x)\ge 1-2\delta$ for $0\le t\le \tau$ and $(c_u+\varepsilon)T+R_{\varepsilon}+(c_{uv}+\varepsilon) t\le |x|\le (c_v-\varepsilon)T-R_{\varepsilon}-(c_{uv}+\varepsilon)t$. In particular, $u(T+\tau,x)\le \delta$, $v(T+\tau,x)\ge 1-\delta$ for $(c_u+\varepsilon)T+R_{\varepsilon}+(c_{uv}+\varepsilon) \tau\le |x|\le (c_v-\varepsilon)T-R_{\varepsilon}-(c_{uv}+\varepsilon)\tau$. Moreover, notice that $(c_u+\varepsilon)(T+\tau)=(c_v-\varepsilon)T-R_{\varepsilon}-(c_{uv}+\varepsilon)\tau$ and hence,  
\begin{equation}\label{tau2delta}
\begin{aligned}
u(t,x)\le 2\delta,\, v(t,x)\ge 1-2\delta,
\end{aligned}
\end{equation}
for any $T\le t\le T+\tau$ and  $(c_u+\varepsilon)T+R_{\varepsilon}+(c_{uv}+\varepsilon) (t-T)\le |x|\le (c_v-\varepsilon)t$, by \eqref{trapcucv}. In particular, 
\begin{equation}\label{taudelta0}
\begin{aligned}
u(T+\tau,x)\le \delta,\, v(T+\tau,x)\ge 1-\delta,
\end{aligned}
\end{equation}
for  $(c_u+\varepsilon)T+R_{\varepsilon}+(c_{uv}+\varepsilon) \tau\le |x|\le (c_v-\varepsilon)(T+\tau)$.

By applying the above arguments to \eqref{taudelta0} again, one can get that \eqref{tau2delta} holds for any $T+\tau\le t\le T+2\tau$ and  $(c_u+\varepsilon)T+2R_{\varepsilon}+(c_{uv}+\varepsilon) (t-T)\le |x|\le (c_v-\varepsilon)t$ and in particular, 
\begin{equation*}
\begin{aligned}
u(T+2\tau,x)\le \delta,\, v(T+2\tau,x)\ge 1-\delta,
\end{aligned}
\end{equation*}
for  $(c_u+\varepsilon)T+2R_{\varepsilon}+2(c_{uv}+\varepsilon) \tau\le |x|\le (c_v-\varepsilon)(T+2\tau)$. By induction, it follows that \eqref{tau2delta} holds
for any $T+(m-1)\tau\le t\le T+m\tau$ and  $(c_u+\varepsilon)T+mR_{\varepsilon}+(c_{uv}+\varepsilon) (t-T)\le |x|\le (c_v-\varepsilon)t$ with $m\in \mathbb{N}$. Clearly, $(m-1)R_{\varepsilon}\le \varepsilon (m-1)\tau\le \varepsilon(t-T)$ for $T+(m-1)\tau\le t\le T+m\tau$. Therefore, \eqref{tau2delta} holds for any $t\ge T$ and  $(c_u+\varepsilon)T+R_{\varepsilon}+(c_{uv}+\varepsilon) (t-T)\le |x|\le (c_v-\varepsilon)t$. By arbitrariness of $\delta$, one has gets \eqref{convergence-cuv-cv}.

This completes the proof.
\end{proof}

\section{Proofs for scenario \textbf{(C2)}.}\label{sec4}
In this section, we prove Theorems \ref{thm2.4} and \ref{ASSset}. We only have to prove Theorem \ref{ASSset} since Theorem \ref{ASSset} is the uniform convergence of Theorem \ref{thm2.4}. We  denote the solution of \eqref{pro1} with scenario \textbf{(C1)} by $(u(t,x),v(t,x))$ in the sequel. 

We can immediately get the following lemma through Lemma~\ref{lemma:sub}.

\begin{lemma}\label{usco1}
	Assume that $U_{\rho}\neq \emptyset$ with 	$\rho>0$ given by Proposition~\ref{proposition1}. Then it holds that 
	\[
	\forall~ c\in (0, c_{uv}), ~~\inf_{x\in U_{\rho}+B_{ct}}
	u(t, x)\to 1.
	\]
\end{lemma}

\begin{proof}
Let $\delta_0>0$, $\rho>0$ and $(\underline{u}_{\rho}(t,x),\overline{v}_{\rho}(t,x))$ be given by Lemma~\ref{lemma:sub}. Take any $x_0\in U_{\rho}$.	By Proposition~\ref{proposition1}, there is $T>0$ such that
$$u(T+t,x_0+x)\ge \underline{u}_{\rho}(t,x) \hbox{ and } v(T+t,x_0+x)\le \overline{v}_{\rho}(t,x),$$
for $t\ge 0$ and $x\in\mathbb{R}^N$. Then, the conclusion follows from Lemma~\ref{lemma:sub}.
\end{proof}

The following lemma is borrowed from \cite[Lemma~4.3]{HR} which concerns the sets $\mathcal{B}(U)$ and $\mathcal{U}(U)$.
 
 \begin{lemma}{\cite{HR}}\label{unlem}
 	Let $U$ be a non-empty subset of 
 	$\mathbb{R}^{N}$ satisfying (\ref{ABcon}) for some $\varrho>0$.
 	Then for every $e\in \mathcal{B}(U)$, there holds
 	\begin{equation}
 		\liminf_{\tau\to+\infty}\frac
 		{\text{d}(\tau e, A)}{\tau}=
 		\inf_{\xi\in \mathcal{U}(U), 
 			\xi\cdot e\ge 0}
 		\sqrt{1-(\xi\cdot e)^{2}}>0,
 	\end{equation}
 	with the convention that the right-hand side of the equality is 1 if there is no $\xi \in \mathcal{U}(U)$ 	satisfying $\xi\cdot e\ge 0$.
 \end{lemma}
 
We establish a set where $(u,v)$ converges to $(0,1)$ by using a family of supersolutions provided by Lemma~\ref{lem1}. 
 
 \begin{lemma}\label{unle1}
 	Assume that 	$U\neq \emptyset$ satisfies \eqref{ABcon}. Assume that there exists $e\in \mathcal{B}(U)$, then for any $w
 	>w(e)$, where $w(e)$ is given by (\ref{unuve}), in the cone
 	\[
 	\mathcal{C}:=\bigcup_{\lambda>1}
 	B(\lambda w e,c_{uv}(\lambda-1)),
 	\]
 	there holds that 
 	\[
 	\sup_{x\in \mathcal{C}} u(t, tx)\to 0 ~\text{and}~
 	\inf_{x\in \mathcal{C}} v(t, tx)\to 1~~\text{as}~
 	t\to +\infty.
 	\]
 \end{lemma}
 \begin{proof}
Since $w>w(e)>0$, there exists a real number $k$ satisfying 
 	\[
 	0<\frac{c_{uv}}{w}<k<\frac{c_{uv}}{w(e)}=\inf_{\xi\in \mathcal{U}(U), 
 		\xi\cdot e\ge 0}
 	\sqrt{1-(\xi\cdot e)^{2}}.
 	\] 
Then, by Lemma \ref{unlem},  there 
 	exists $\tau_{1}>0$ such that
 	\[
 	\forall \tau\ge \tau_{1}, ~\text{dist}~(\tau e, U)\ge 
 	k\tau.
 	\] 
 	Take $\varepsilon\in (0,(kw-c_{uv})/2)$. Let $\delta_0$ and $R_{\varepsilon}$ given by Lemma \ref{lem1}. 
For $\tau>\tau_1 +R_{\varepsilon}/k$, notice that $B(\tau e, k\tau) \subset \mathbb{R}^N\setminus U$. Then, we can compare $(u(T,\tau e +x),v(T,\tau e +x))$ with $\overline{u}_{k\tau}(T, x)$ and 
 	$\underline{v}_{k\tau}(T, x)$ given by Lemma~\ref{lem1}. Therefore by the 
 	comparison principle, there holds that 
\begin{equation}\label{T2delta0}
u(T,\tau e +x)\le \overline{u}_{k\tau}(T, x)\le 2\delta_0,\, v(T,\tau e +x)\ge \underline{v}_{k\tau}(T, x)\ge 1-2\delta_0, 
\end{equation}
 	for $0\le T\le	\frac{k\tau-R_{\varepsilon}}{c_{uv}+\varepsilon}$ and $|x|\le k\tau-R_{\varepsilon}-(c_{uv}+\varepsilon)T$. For any $x\in \mathcal{C}$, there is $\lambda>1$ such that $x\in B(\lambda w e, c_{uv}(\lambda-1))$. Let $\tau=\lambda w t$. Obviously, $\frac{k\tau-R_{\varepsilon}}{c_{uv}+\varepsilon}>t$ for large $t$. Hence, 
 	$$|xt-\lambda \omega e t|\le  c_{uv}(\lambda -1)t\le \lambda (kw -\frac{2\varepsilon}{\lambda})t -c_{uv}t\le k\lambda w t -R_{\varepsilon} -(c_{uv}+\varepsilon)t,$$
 	for large $t$. By taking $T= t$ in \eqref{T2delta0}, one has that $u(t,xt)\le 2\delta_0$ and $v(t,xt)\ge 1-2\delta_0$. By arbitrariness of $\delta_0$, one gets the conclusion.
 \end{proof}

\begin{proof}[Proof of Theorem \ref{ASSset}]
Once we have Lemma \ref{usco1} and Lemma \ref{unle1} at hand, the rest proof is similar as that of  \cite[Theorem~2.2]{HR}. We finish the proof for completeness.

The proof for the geometric expression \eqref{expressionW} can be referred to the first paragraph of the proof of \cite[Theorem~2.2]{HR}. 

 By \eqref{ABcon}, it is clear 
	that 
	$
	\mathcal{U}(U)=\mathcal{U}(U_{\rho})$. Consider a 
	compact set $\mathcal{C}$ contained in 
	$\mathcal{W}$. For any $\xi\in
	\mathcal{U}(U)=\mathcal{U}(U_{\rho})$
	(if it exists), any $\tau>0$ and 
	$0<c<c'<c_{uv}$, from
	the defination of $\mathcal{U}(U_{\rho})$ in 
	Defination \ref{deuv}, it implies
	\[
	\frac{1}{t}\text{dist}\left( 
	t\tau\xi, U_{\rho}\right)\to +\infty,
	\]
	thus one gets $B(t\tau\xi,c't)
	\subset U_{\rho}+B_{ct}$ for
	sufficiently large $t$. Then by 
	Lemma \ref{usco1}, it yields
	\begin{equation}\label{convergence-c'}
		\inf_{x\in B(t\tau\xi,c't)}
		u(t, tx)\to 1,~~
		\sup_{x\in B(t\tau\xi,c't)}v(t, tx)
		\to 0.
	\end{equation}
	The limitation holds even for
	$\tau=0$  (without any reference 
	to $\xi$). Moreover, the expression \eqref{ABcon}
	implies that  any $x\in 
	\mathcal{W}$ is included either in 
	$B(0,c_{x}')$ or in $B(
	\tau_{x}\xi_{x},c_{x}')$ for certain 
	$c_{x}'\in (0, c_{uv})$ 
	, $\xi_{x}\in \mathcal{U}(U)$
	and $\tau_{x}>0$. Since compact set 
	$\mathcal{C}$ can be covered by 
	a finite number of such balls due to 
	its compactness, 	the first line of \eqref{28} is 
	concluded by \eqref{convergence-c'}.
	
	Now consdier a compact set $\mathcal{C}$ included in $\mathbb{R}^{N}\setminus\overline {\mathcal{W} }$. Take any point $y\in \mathcal{C}$ and let $e=y/|y|$. Then, they satisfy $w(e)<|y|<+\infty$ and hence, $e\in \mathcal{B}(U)$ because of the convention \eqref{unge}. As a	consequence, we can apply  Lemma 
	\ref{unle1} with $w\in (w(e), |y|)$. For 
	any small $\epsilon\in (0,1)$, we deduce
	that there is an open neighborhood $\mathcal{C}_{y}$ of $y$ and some $t_{y}>0$ such that 
	\[
	u(t, tx)<\epsilon, ~v(t, tx)>1-\epsilon,\quad \hbox{for }\forall ~t>t_{y}, ~x\in \mathcal{C}_{y}.
	\]
	As a conclusion, by a covering arguemnet, 
	we can find $t_{C}>0$ such that 
	\[
	u(t, tx)<\epsilon, ~v(t, tx)>1-\epsilon,\quad	\hbox{for }\forall ~t>t_{C}, ~x\in \mathcal{C}.
	\]
This proves the second line of \eqref{28}.
\end{proof}

\begin{proof}[Proof of Theorem~\ref{thm2.4}]
The first line of \eqref{26} is deduced exactly from the first line of \eqref{28}. For the second line of \eqref{26}, one only needs to consider the directions $e\in\mathcal{B}(U)$ since one necessarily has $w(e)<+\infty$. Then, such a limit immediately
follows by applying Lemma \ref{unle1} with $w\in (w(e),c)$.
\end{proof}

\section*{Acknowledgments}

This work is supported by the fundamental research funds for the central universities and NSF of China (No. 12101456, No. 12471201).
\vskip 0.3cm

\noindent
\textbf{Data availability} No data were generated or used during the study.
\vskip 0.3cm

\noindent
\textbf{Conflict of interest} On behalf of all authors, the corresponding author states that there is no Conflict of interest.


\end{document}